\documentclass{amsart}

\usepackage{amssymb, amscd, graphicx, amsmath, setspace, color, enumerate, tikz-cd}
\usepackage[all]{xy}

\newtheorem{theorem}{Theorem}[section]

\newtheorem{corollary}[theorem]{Corollary}

\newtheorem{lemma}[theorem]{Lemma}

\newtheorem{proposition}[theorem]{Proposition}

\theoremstyle{remark}

\newtheorem{remark}[theorem]{Remark}

\newcommand{\CA}{\mathcal{A}}

\newcommand{\CC}{\mathcal{C}}

\newcommand{\CI}{\mathcal{I}}

\newcommand{\calT}{\mathcal{T}}

\newcommand{\Hom}{\mathrm{Hom}}

\newcommand{\Ker}{\mathrm{Ker}}

\newcommand{\End}{\mathrm{End}}

\newcommand{\fp}{\mathrm{fp}}

\newcommand{\Add}{\mathrm{ Add}}

\newcommand{\Prod}{\mathrm{ Prod}}

\newcommand{\Modr}{\mathrm{ Mod}\text{-}}

\title[]{Comparing $\Add(M)$ with $\Prod(M)$}

\author{Simion Breaz, Cristian Rafiliu}

\address{"Babe\c s-Bolyai" University, Faculty of Mathematics and Computer Science, Str. Mihail Kog\u alniceanu 1, 400084, Cluj-Napoca, Romania}

\email[Simion Breaz]{simion.breaz@ubbcluj.ro}

\email[Cristian Rafiliu]{cristian.rafiliu@ubbcluj.ro}

\begin{document}
	
	\begin{abstract}
We present characterizations for the inclusions $\Add(M)\subseteq \Prod(M)$ and $\Prod(M)\subseteq \Add(M)$ in locally finitely presented categories and in compactly generated triangulated categories. As applications, we describe the situations when the classes of the form $\Prod(M)$ and $\Add(M)$ are (pre)covering, respectively (pre)enveloping.  
	\end{abstract}
	
	
	\maketitle
	
\section{Introduction}

Let $\CA$ be an additive category with direct sums and products. If $M\in \CA$ then $\Add(M)$ and $\Prod(M)$ denote the class of all direct summands of direct sums, respectively of direct products of copies of $M$. If $\CA$ is a module category the situations where one of these classes is contained in the other, investigated in \cite{angeleri2003}, \cite{angeleri2002} and \cite{Kr-Sa}, lead to useful properties of $M$. For instance, a module $M$ has the property $\Prod(M)\subseteq \Add(M)$, such modules were introduced in \cite{Kr-Sa} under the name \textit{product-complete}, if and only if $\Add(M)$ is an (pre)enveloping class (\cite[Theorem 5.1]{angeleri2003}, \cite[Theorem 2.3]{Kr-Sa}). They play an important role in tilting theory \cite{AnSaTr}. Recently, this property was also used in locally finitely presented categories, \cite{Krause2022}, and in compactly generated triangulated categories, \cite{HM24}.    
On the other hand, the study and the applications of the inclusion $\Add(M)\subseteq \Prod(M)$ depend on some set-theoretic hypotheses. In the absence of $\omega$-measurable cardinals, it was proved that $\Add(M)\subseteq \Prod(M)$ if and only if $M$ is $\Sigma$-pure-injective module (see \cite{Breaz2015} and \cite{Saroch2015}). However, \v Saroch proved in \cite{Saroch2015} that if we assume that there are some large cardinals then in every module category there exists a free module $F$ such that $\Add(F)\subseteq \Prod(F)$. Such modules are used in \cite{Co-Sa} and \cite{Sa-Tr}.

Our first aim is to study these inclusions for locally finitely presented categories. To do this, we will start with an analysis of Chase's Lemma \cite{Chase1962}, since this is an important tool in proving characterizations for $\Sigma$-pure-injective modules, see \cite{Zimmermann-Huisgen1979}. The usual statement of this lemma is about maps defined on countable direct products. We also need a version when the domains are direct products indexed by non $\omega$-measurable cardinals, a generalization of \cite[Theorem 2]{Dugas&Zimmermann-Huisgen1981}. 
We will use this to prove in Theorem \ref{thm:lfp-Add-subset-Prod} that an object $M$, with $\Hom(X,M)$ non $\omega$-measurable for all finitely presented objects $X$, is $\Sigma$-pure-injective if and only if $\Add(M)\subseteq \Prod(M)$. As an application, in Corollary \ref{cor:angeleri-6.11} we prove that if we assume that there are no $\omega$-measurable cardinals, then the class $\Prod(M)$ is (pre)covering if and only if $M$ is $\Sigma$-pure-injective. This leads to the conclusion that, if we assume that there are no $\omega$-measurable cardinals, Enochs' Conjecture is true for classes of the form $\Prod(M)$ in locally finitely presented categories, Corollary \ref{cor:enochs-prod-lfp}. 
A characterization for $\Prod(M)\subseteq \Add(M)$ 
is presented in Theorem \ref{thm:productcompleteA}. This theorem extends \cite[Proposition 5.2]{angeleri2002} (see also \cite[Proposition 4.8]{Angeleri2000}) from module categories to locally finitely presented categories. Moreover, $\Add(M)$ is (pre)enveloping if and only if $M$ is product-complete (Corollary \ref{cor:add-env}). 

In the last part of the paper we will use the results obtained for locally finitely presented categories to approach the same problems in compactly generated triangulated categories. As expected, we have similar characterizations. We include in Corollary \ref{cor:Enochs-triangulated} a positive answer for classes of the form $\Prod(M)$ to a triangulated version of Enochs's conjecture. In Theorem \ref{thm:triang-Prod-subset-Add} we also give a triangulated version of \cite[Proposition 12.3.7]{Krause2022}.
    
\section{A Chase-type Lemma}
	
A tuple $$( (A_i)_{i\in I}, (B_j)_{j\in J}, \varphi, ((A_{i,k})_{k\in\mathbb{N}})_{i\in I}, ((B_{j,k})_{k\in\mathbb{N}})_{j\in J})$$ is called a \textit{Chase system} if
\begin{itemize}
    \item $(A_i)_{i\in I}$ and $(B_j)_{j\in J}$ are two familes of abelian groups,
    \item $\varphi:\prod_{i\in I} A_i\to \bigoplus_{j\in J} B_j$ is a morphism of abelian groups,
    \item for all $i\in I$, $(A_{i,k})_{k\in\mathbb{N}}$ is a descending family of subgroups of $A_i$, 
    \item for all $j\in J$, $(B_{j,k})_{k\in\mathbb{N}}$ is a descending family of subgroups of $B_j$, 
    \item for all $k\in \mathbb{N}$, $\varphi(\prod_{i\in I} A_{i,k})\subseteq \bigoplus_{j\in J} B_{j,k}$.
\end{itemize}

The following result is a version for Chase's Lemma.
    
\begin{lemma}\label{lem:countable-Chase}
Let 
$$( (A_i)_{i\in \mathbb{N}}, (B_j)_{j\in J}, \varphi, ((A_{i,k})_{k\in\mathbb{N}})_{i\in \mathbb{N}}, 
((B_{j,k})_{k\in\mathbb{N}})_{j\in J})$$ be a Chase system. Then there exists a finite subset $J'\subseteq J$ and two positive integers $m_0$ and $n_0$ such that
$$\textstyle\varphi\left( \prod_{i\geq m_0} A_{i,n_0}\right) \subseteq \bigoplus_{j\in J'} B_j + \bigcap_{n\in\mathbb{N}} \left( \bigoplus_{j\in J} B_{j,n}\right).$$
	\end{lemma}
	
\begin{proof} We observe that it is enough to prove the conclusion for $m_0=n_0$. 
Following the steps of Chase \cite{Chase1962}, we assume by contradiction that the conclusion is false. 

In particular, for every $n\in \mathbb{N}$ and for every finite subset $J'\subseteq J$ we have $\varphi\left( \prod_{i\geq n} A_{i,n}\right) \nsubseteq \bigoplus_{j\in J'}B_j+ \bigcap_{n\in\mathbb{N}} \left( \bigoplus_{j\in J} B_{j,n}\right)$.
We will construct a strictly ascending sequence $(n_k)_{k\in \mathbb{N}}$ of positive integers, a sequence $(j_k)_{k\in \mathbb{N}}$ of elements of $J$ and a sequence $(x_k)_{k\in\mathbb{N}}$ with $x_k\in \prod_{i\geq n_k}A_{i,n_k}$, such that for any $k\in \mathbb{N}$ we have:
		\begin{itemize}
			\item[(a)] $q_{j_k}\varphi(x_k)\notin B_{j_k,n_{k+1}}$;
			\item[(b)] $q_{j_k}\varphi(x_l) = 0$, for all $l<k$,
		\end{itemize}
where $q_{j_k}:\bigoplus_{j\in J} B_j\to B_{j_k}$ are the natural projections.

For $k=0$ take $n_0=0$. It follows that there exists $x_0\in \prod_{i\geq 0}A_{i,0}$ such that $\varphi(x_0)\notin \bigcap_{n\in\mathbb{N}} \left( \bigoplus_{j\in J} B_{j,n}\right)$. This implies that there exists $n_1\in \mathbb{N}$ such that $\varphi(x_0)\notin \bigoplus_{j\in J} B_{j,n_1}$, hence there exists $j_0\in J$ such that $q_{j_0}\varphi(x_0)\notin B_{j_0,n_{1}}$. 

Let $J'=\{j\in J\mid q_j\varphi(x_0)\neq 0\}$. Since $\varphi\left( \prod_{i\geq n_1} A_{i,n_1}\right)$ is not contained in $\bigoplus_{j\in J'}B_j+ \bigcap_{n\in\mathbb{N}} \left( \bigoplus_{j\notin J'} B_{j,n}\right)\subseteq \bigoplus_{j\in J'}B_j+ \bigcap_{n\in\mathbb{N}} \left( \bigoplus_{j\in J} B_{j,n}\right),$ there exists $x_1\in \prod_{i\geq n_1}A_{i,n_1}$, $j_1\notin J'$ and $n_2\in\mathbb{N}$ such that  $\varphi(x_1)\notin B_{j_1,n_2}$. Since the chains of subgroups are descending, we observe that we can take $n_2\geq n_1$. 

Consequently, the elements $x_0,x_1, n_1, n_2, j_0,$ and $j_1$ satisfy the conditions (a) and (b). We can continue in the same manner, and we obtain the required sequences by induction.

We define $x = \sum_{k\in \mathbb{N}}x_k\in \prod_{i\in\mathbb{N}} A_i$ (because $x_k\in \prod_{i\geq n_k}A_i$, for every $i\in\mathbb{N}$ the $i$-th components of the elements $x_k$ are almost all $0$). 

Observe that for $l>k$ we have $x_l\in \prod_{i\geq n_l} A_{i,n_l}\subseteq \prod_{i\geq n_l} A_{i,n_{k+1}}\subseteq \prod_{i\in\mathbb{N}} A_{i,n_{k+1}}$, so $\varphi(x_l)\in \varphi(\prod_{i\in\mathbb{N}} A_{i,n_{k+1}})\subseteq \bigoplus_{j\in J} B_{j,n_{k+1}}.$
It follows that for all $k\in \mathbb{N}$ we obtain
		$$q_{j_k}\varphi(x) = q_{j_k}\varphi(\sum_{l<k}x_l) + q_{j_k}\varphi(x_k) + q_{j_k}\varphi(\sum_{l>k}x_l)\neq 0,$$
because $q_{j_k}\varphi(\sum_{l<k}x_l)=0$, $q_{j_k}\varphi(x_k)\notin B_{j_k,n_{k+1}}$, and  $q_{j_k}\varphi(\sum_{l>k}x_l)\in B_{j_k,n_{k+1}}$.
This contradicts the fact that $\varphi(x)\in \bigoplus_{j\in J}B_j$, and the proof is complete.
\end{proof}
	
\begin{remark}
Similar results, proved with similar techniques, are also provided in \cite[Lemma 4.8]{Bennett-Tennenhaus2023}, in the language of model theory, and in \cite[Lemma 4]{Krause2021} for subgroups of finite definition. 
\end{remark}
    	
In the following, we will give a more general version of the previous version of Chase's Lemma. It extends similar results proved in \cite{Breaz2015} and \cite{Dugas&Zimmermann-Huisgen1981}, where the chains of subgroups are values of some descending chains of subfunctors of the identity functor. 

A cardinal $\kappa=|I|$ (or the set $I$) is called \textit{$\omega$-measurable} if it is uncountable and there exists a countably-additive, non-trivial, $\{0,1\}$-valued measure $\mu$, whose $\sigma$-algebra is the power set of $I$ such that $\mu(I) = 1$ and $\mu(\{x\}) = 0$ for all $x\in I$. We refer to \cite{Ek-Me} for details about (non) $\omega$-measurable cardinals.
The next lemma gives us an easy way to prove that a cardinal is $\omega$-measurable. The reader can find proofs for this in the proofs of \cite[Theorem 2]{Dugas&Zimmermann-Huisgen1981} and \cite[Proposition 26]{BZ}.

\begin{lemma}\label{lem:measurable}
Let $I$ be an uncountable set. If there exists a subset $\CI$ in the power-set of $I$ such that
\begin{enumerate}[{\rm (i)}]
\item $I\notin \CI$,
    \item if $F\subseteq I$ is finite then $F\in \CI$,
    \item $\CI$ is closed with respect to finite unions,
    \item If $X\in \CI$ and $Y\subseteq X$ then $Y\in \CI$,
    \item if $(X_n)_{n\geq 0}$ is a countable family of pairwise disjoint subsets of $I$ then there exists $n_0$ such that $\bigcup_{n\geq n_0}X_n\in \CI$,  
\end{enumerate}
then $|I|$ is $\omega$-measurable. 
\end{lemma}
    
\begin{lemma}\label{lemma2}\label{lem:chase-uncountable}
Suppose that the set $I$ is not $\omega$-measurable. If 
$$( (A_i)_{i\in I}, (B_j)_{j\in J}, \varphi, ((A_{i,k})_{k\in\mathbb{N}})_{i\in I}, ((B_{j,k})_{k\in\mathbb{N}})_{j\in J})$$ is a Chase system, then there are two finite subsets $I'\subseteq I, J'\subseteq J$ and a positive integer $n_0$ such that
$$\textstyle \varphi\left( \prod_{i\in I\backslash I'} A_{i,n_0}\right) \subseteq \bigoplus_{j\in J'} B_j + \bigcap_{n\in\mathbb{N}} \left( \bigoplus_{j\in J} B_{j,n}\right).$$
\end{lemma}
	
\begin{proof}
If $I$ is finite, the conclusion is obvious. For the case $I$ countable, we apply Lemma \ref{lem:countable-Chase}.

Suppose that $|I|>\aleph_0$. 
We consider the set $\mathcal{I}$ of all subsets $T\subseteq I$ with the property that the conclusion of our lemma is valid for the Chase system 
        $$\textstyle ( (A_i)_{i\in T}, (B_j)_{j\in J}, \varphi_{|\prod_{i\in T}A_i}, ((A_{i,k})_{k\in\mathbb{N}})_{i\in T}, ((B_{j,k})_{k\in\mathbb{N}})_{j\in J}),$$
where $\varphi_{|\prod_{i\in T}A_i}$ denotes the restriction of $\varphi$ to $\prod_{i\in T}A_i$. We will prove that $\CI$ satisfies the conditions (ii)--(v) from Lemma \ref{lem:measurable}. 

It is not hard to verify the conditions (ii), (iii) and (iv) listed in Lemma \ref{lem:measurable}. For (v), let $(I_n)_{n\in\mathbb{N}}$ be a family of pairwise disjoint subsets of $I$. We observe first that we can use the identification $$\textstyle\prod_{i\in\bigcup_{n\in\mathbb{N}} I_n} A_i =  \prod_{n\in\mathbb{N}}\left( \prod_{i\in I_n} A_i\right)$$
to obtain the Chase system
		$$\textstyle\left( \left( \prod_{i\in I_n} A_i\right)_{n\in \mathbb{N}}, (B_j)_{j\in J}, \varphi', \left( \left( \prod_{i\in I_n} A_{i,k}\right)_{k\in\mathbb{N}}\right)_{n\in \mathbb{N}}, ((B_{j,k})_{k\in\mathbb{N}})_{j\in J}\right), $$
where $\phi'$ is the restriction of $\varphi$ to $\prod_{n\in\mathbb{N}}\left( \prod_{i\in I_n} A_i\right)$.
We apply Lemma \ref{lem:countable-Chase}, and it follows that $\cup_{n\geq m_0} I_n\in \mathcal{I}$.
		
Since the set $I$ is not $\omega$-measurable, it follows that $I\in \CI$, and the proof is complete.
	\end{proof}

\begin{remark}
Lemma \ref{lem:chase-uncountable} is not valid if $I$ is $\omega$-measurable. A counterexample is constructed in \cite[Example 3]{Dugas&Zimmermann-Huisgen1981}.   
\end{remark}
	
Chase's Lemma is often used to provide descending chain conditions. In this paper, we will use the following proposition. 

\begin{proposition}\label{proposition}
Let $$\textstyle ((A_i)_{i\in I}, (B_j)_{j\in J}, \varphi: \prod_{i\in I} A_i \to \bigoplus_{j\in J} B_j, ((A_{i,k})_{k\in\mathbb{N}})_{i\in I},  ((B_{j,k})_{k\in\mathbb{N}})_{j\in J})$$ be a Chase system. 

If
\begin{enumerate}[{\rm (i)}]
    \item $I$ is a set of non-$\omega$-measurable cardinality, 
    \item $J$ an infinite set with $|J|>|A_i|$ for all $i\in I$, 
    \item $\varphi$ is an epimorphism, and 
    \item the restrictions $\varphi|_{\prod_{i\in I} A_{i,k}}: \prod_{i\in I} A_{i,k}\to \bigoplus_{j\in J} B_{j,k}$ are epimorphisms for all $k\in\mathbb{N}$,\end{enumerate} then there exists a positive integer $n_0$ and an infinite subset $L\subseteq J$ such that $B_{j,n}=B_{j,n_0}$ for all $j\in L$ and all $n\geq n_0$.
\end{proposition}
		
\begin{proof}
Applying Lemma \ref{lemma2} for the Chase system
$$((A_i)_{i\in I}, (B_j)_{j\in J}, \varphi, ((A_{i,k})_{k\in\mathbb{N}})_{i\in I}, ((B_{j,k})_{k\in\mathbb{N}})_{j\in J}),$$ we obtain that there are two finite subsets $I'\subseteq I, J'\subseteq J$ and a positive integer $n_0$ such that
$$\textstyle\varphi\left( \prod_{i\in I\backslash I'} A_{i,n_0}\right) \subseteq \bigoplus_{j\in J'} B_j + \bigcap_{n\in\mathbb{N}} \left( \bigoplus_{j\in J} B_{j,n}\right).$$
		
Following almost the same algorithm like in \cite[Proposition 2.3]{Breaz2015}, we obtain that
$$\textstyle\varphi\left( \prod_{i\in I\backslash I'} A_{i,n_0}\right) \subseteq \bigoplus_{j\in J'} B_{j,n_0} + \bigoplus_{j\in J\backslash J'} \left( \bigcap_{n\in\mathbb{N}} B_{j,n}\right)$$
and, because $\varphi|_{\prod_{i\in I} A_{i,n_0}}:\prod_{i\in I} A_{i,n_0}\to \bigoplus_{j\in J} B_{j,n_0}$ is an epimorphism, it induces an epimorphism of abelian groups:
$$\bar{\varphi}: \frac{\prod_{i\in I} A_{i,n_0}}{\prod_{i\in I\backslash I'} A_{i,n_0}}\to \frac{\bigoplus_{j\in J} B_{j,n_0}}{\bigoplus_{j\in J'} B_{j,n_0} + \bigoplus_{j\in J\backslash J'} \left( \bigcap_{n\in\mathbb{N}} B_{j,n}\right)}\cong \bigoplus_{j\in J\backslash J'}\frac{B_{j,n_0}}{\bigcap_{n\in\mathbb{N}} B_{j,n}}.$$
Using (ii), it follows that $$\textstyle\left| \bigoplus_{j\in J\backslash J'}\frac{B_{j,n_0}}{\bigcap_{n\in\mathbb{N}} B_{j,n}} \right| \leq \left| \prod_{i\in I'} A_{i,n_0} \right|< |J|,$$ 
hence the set $L=\{j\in J\backslash J' | \bigcap_{n\in\mathbb{N}} B_{j,n} = B_{j,n_0}\}$ is infinite. 
\end{proof}
	
\section{Comparing $\Add(M)$ with $\Prod(M)$ in locally finitely presented additive categories}

In this section $\mathcal{A}$ will be a locally finitely presented category. The subcategory of all finitely presented objects from $\CA$ is denoted by $\fp\mathcal{A}$. We refer to \cite{Krause2022} for the main properties of these categories. In particular, $\CA$ is idempotent complete, with cokernels, filtered colimits (hence, with direct sums) and direct products.  

\subsection{Purity in locally finitely presented categories} We recall for the reader's convenience some basic information. According to \cite[Lemma 12.1.4]{Krause2022} there is a fully faithful functor $$ev:\mathcal{A}\to \textbf{P}(\mathcal{A}),$$ where $\mathbf{P}(A)$ is a locally finitely presented Grothendieck category, called \textit{the purity category of $\mathcal{A}$}, and $ev$ commutes with all filtered colimits, direct products and cokernels. A sequence of morphisms $$0\to K\overset{\alpha}\to L\overset{\beta}\to M\to 0$$ from $\CA$ is a \textit{pure-exact} sequence if for every $X\in \fp\CA$ the sequence of abelian groups $$0\to \Hom(X,A)\overset{\Hom(X,\alpha)}\longrightarrow \Hom(X,L)\overset{\Hom(X,\beta)}\longrightarrow \Hom(X,M)\to 0$$ is a short exact sequence. In this case, we say that $\alpha$ is a \textit{pure mononomorphism} and $\beta$ is a \textit{pure epimorphism}. If $L\in\CA$, an object $K$ is a \textit{pure subobject} of $L$ if there exists a pure monomorphism $K\to L$, and an object $M$ is a \textit{pure quotient} of $L$ if there exists a pure epimorphism $L\to M.$ For further use, let us remark that from \cite[Lemma 12.1.6 (2)]{Krause2022} it follows that these notions coincide with those used in \cite{LV}.   

\begin{lemma}\label{lem:basic-purity} The following are true in a locally finitely presented category $\CA$.
\begin{enumerate}[{\rm a)}]
    \item A sequence of morphisms $0\to K\overset{\alpha}\to L\overset{\beta}\to M\to 0$ is pure-exact if and only if $0\to ev(K)\overset{ev(\alpha)}\longrightarrow ev(L)\overset{ev(\beta)}\longrightarrow ev(M)\to 0$ is a short exact sequence in $\mathbf{P}(A)$.
    \item A morphism $\alpha:L\to K$ in $\CA$ is a pure monomorphism if and only if $ev(\alpha)$ is a monomorphism. 
    \item A sequence of morphisms $0\to K\overset{\alpha}\to L\overset{\beta}\to M\to 0$ is pure-exact if and only if $\alpha$ is a pure monomorphism and $\beta$ is the cokernel of $\alpha$. 
    \item A composition of two pure monomorphisms is a pure monomorphism.
    \item A direct product (direct sum) of pure-exact sequences is pure-exact. 
\end{enumerate}
\end{lemma}

\begin{proof}
a) This is proved in \cite[Lemma 12.1.6]{Krause2022}.

b) See \cite[Lemma 12.1.7]{Krause2022}.

c) Suppose that $0\to K\overset{\alpha}\to L\overset{\beta}\to M\to 0$ is pure-exact. Let $L\overset{\nu}\to N$ be the cokernel of $\alpha$. Since $ev$ preserves the cokernels, it follows that $ev(\nu)$ is the cokernel of $ev(\alpha)$. Since the sequence is pure-exact, it follows that $ev(\beta)$ is the cokernel of $ev(\alpha)$. Then there exists an isomorphism $\overline{\gamma}:ev(M)\to ev(N)$ such that $ev(\nu)=\overline{\gamma}ev(\beta).$ But $ev$ is fully faithful, hence there exists an isomorphism  $\gamma:M\to N$ such that $\nu=\gamma\beta$, hence $\beta$ is a cokernel for $\alpha$. 

Conversely, if $\alpha$ is a pure monomorphism and $\beta$ is the cokernel of $\alpha$, since $ev$ preserves the cokernels, it follows that $0\to ev(K)\overset{ev(\alpha)}\longrightarrow ev(L)\overset{ev(\beta)}\longrightarrow ev(M)\to 0$ is a short exact sequence.

d) This follows from b).

e) This is true since the covariant functors $\Hom(X,-)$ (with $X$ finitely presented) commute with respect to direct products (direct sums) and in the category of all abelian groups a direct product (direct sum) of short exact sequences is a short exact sequence. 
\end{proof}

An object $N$ is \textit{pure-injective} if for every pure monomorphism $\alpha$ the morphism $\Hom(\alpha,N)$ is surjective. It follows from \cite[Lemma 12.1.8]{Krause2022} that $ev$ gives an equivalence between the subcategory of pure-injective objects of $\CA$ and the subcategory of injective objects from $\mathbf{P}(\CA)$. 
    
\subsection{Subgroups of finite definition} 

Every morphism $b:X\to Y$ in $\fp\mathcal{A}$ induces an additive functor $$\Hom(Y,-)b:\mathcal{A}\to \textbf{Ab},$$ defined in the following way: 
\begin{itemize}
    \item $\Hom(Y,M)b = \{f\in \Hom(X,M) | f = \varphi b, \text{ for some }\varphi\in \Hom(Y,M)\}$ for all $M\in \mathcal{A}$, and
    \item $\Hom(Y,f)b:\Hom(Y,M)b\to \Hom(Y,N)b, 
gb \mapsto fgb,$ for every morphism $f:M\to N$.
\end{itemize}
The sets $\Hom(Y,M)b$ are called the \textit{subgroups of finite definition of $\Hom(X,M)$}.

\begin{remark} Let $b:X\to Y\in \fp\mathcal{A}$. Then:
		\begin{enumerate}
			\item $\Hom(Y,-)b:\mathcal{A}\to \textbf{Ab}$ is an additive subfunctor of the covariant functor $\Hom(X,-):\mathcal{A}\to \textbf{Ab}$;
			
			\item $\Hom(Y,-)b$ commutes with respect to direct products and direct sums i.e. for any family $(M_i)_{i\in I}$ of objects in $\mathcal{A}$ we have natural isomorphisms $$\textstyle\Hom(Y,\prod_{i\in I} M_i)b \simeq \prod_{i\in I} \Hom(Y,M_i)b$$
		and 	
			$$\textstyle\Hom(Y,\bigoplus_{i\in I}M_i)b \simeq \bigoplus_{i\in I}\Hom(Y,M_i)b.$$
			
			\item For any morphism $f:\prod_{i\in I} M_i\to \bigoplus_{j\in J} N_j$ in $\mathcal{A}$ we can view the morphism $\Hom(Y,f)b$ as a morphism
			$$\textstyle\Hom(Y,f)b :\prod_{i\in I} \Hom(Y,M_i)b\to \bigoplus_{j\in J} \Hom(Y,N_j)b.$$
		\end{enumerate}
\end{remark}

\begin{remark}
The subgroups of finite definition are also used in compactly generated triangulated categories in \cite{Bennett-Tennenhaus2023} and \cite{Garkusha&Prest2005}. In particular, it is shown that they coincide with the $pp$-definable subgroups used in model theory.   
\end{remark}

\subsection{$\Sigma$-pure-injective objects and the inclusion $\Add(M)\subseteq \Prod(M)$}
	
For an object $M$ in $\mathcal{A}$ we denote by $\Prod(M)$ and $\Add(M)$ the class of all direct summands in direct products, respectively direct sums of copies of $M$. 

An object $M$ is called \textit{$\Sigma$-pure-injective} if, for any set $J$, the direct sum $M^{(J)}$ is pure-injective.
We mention some of the characterizations about $\Sigma$-pure-injective objects given in \cite{Krause2022}:
	
\begin{theorem} \cite[Theorem 12.3.4]{Krause2022}\label{theoremK} Let $M$ an object in $\mathcal{A}$. The following are equivalent:
		\begin{enumerate}[{\rm (i)}]
			\item $M$ is $\Sigma$-pure-injective.
			
			\item Every product of copies of $M$ is isomorphic to a direct sum of indecomposable objects with local endomorphism rings.
			
			\item There is an object $Y$ such that every product of copies of $M$ is a pure subobject of a direct sum of copies of $Y$.
			
			\item For any $X$ of $\fp\mathcal{A}$ all descending chains of subgroups of finite definition of $\Hom(X,M)$ 
            are stationary.
		\end{enumerate}
\end{theorem}

\begin{corollary}\label{cor:closure-abelian} 
Suppose that $M$ is a $\Sigma$-pure-injective object from $\CA$. The following are true:
\begin{enumerate}[{\rm a)}]
\item For every set $I$ the object $M^I$ is $\Sigma$-pure-injective. 

\item If $\alpha:K\to M$ is a pure monomorphism then $K$ is $\Sigma$-pure-injective and $\alpha$ is a splitting monomorphism.

\item Every pure epimorphism $\beta: M\to N$ is a splitting epimorphism.   
\end{enumerate}
\end{corollary}
\begin{proof}
 a) This follows from Theorem \ref{theoremK}(iii). 
 
 b) Let $Y$ be an object as in Theorem \ref{theoremK}(iii). If $I$ is a set, it follows from Lemma \ref{lem:basic-purity} that $\alpha^I:K^I\to M^I$ is a pure monomorphism, hence $K^I$ is a pure subobject of a direct sum of copies of $Y$. The last statement is true since there exists $\alpha':M\to K$ such that $\alpha'\alpha=1_K$ and $\CA$ is idempotent complete.

 c) This follows from the fact that $\beta$ is the cokernel of a splitting monomorphism in an idempotent complete category.
\end{proof}
    
In order to characterize the inclusion $\Add(M)\subseteq \Prod(M)$ we will follow the similar ideas from \cite{Breaz2015}.
	
\begin{theorem} \label{thm:lfp-Add-subset-Prod}
Let $M$ be an object of $\mathcal{A}$ such that for every $X\in \fp\mathcal{A}$ the cardinal of $\Hom(X,M)$ is not $\omega$-measurable. The following are equivalent:
		\begin{enumerate}[{\rm (i)}]
			\item $\Add(M)\subseteq \Prod(M)$.
			\item $M$ is $\Sigma$-pure-injective.
		\end{enumerate}
\end{theorem}
	
\begin{proof}
(i)$\Rightarrow$(ii) Suppose that $X\in \fp\mathcal{A}$ and that we have a descending chain of subgroups of $\Hom(X,M)$ of the form $(\Hom(Y_k,M)b_k)_{k\in\mathbb{N}}$, where $b_k:X\to Y_k$ are morphism in $\fp\mathcal{A}$.
		
From (i) it follows that for every set $J$ there exist a set $I$ and a split epimorphism
		$$\textstyle f:\prod_{i\in I} M\to \bigoplus_{j\in J} M.$$
Since $\Hom(Y_k,-)b_k$ are additive functors, it follows that all morphisms $$\textstyle\Hom(Y_k,f)b_k: \Hom(Y_k, \prod_{i\in I} M)b_k\to \Hom(Y_k, \bigoplus_{j\in J} M)b_k$$ are split epimorphisms.
        %
		
We apply the Proposition \ref{proposition} for a set $J$ with $|J|>|\Hom(X,M)|$, the groups $A_i = B_j = \Hom(X,M)$, the homomorphism $$\varphi =\textstyle \Hom(X,f):\prod_{i\in I} \Hom(X,M)\to \bigoplus_{j\in J}\Hom(X,M),$$
and the sequences $$(A_{i,k})_{k\in\mathbb{N}} = (B_{j,k})_{k\in\mathbb{N}} = ((\Hom(Y_k,M)b_k)_{k\in\mathbb{N}}$$ of subgroups of $\Hom(X,M)$. 

It follows that $(\Hom(Y_n,M)b_n)_{n\in\mathbb{N}}$ is stationary. From Theorem \ref{theoremK} we obtain that $M$ is $\Sigma$-pure-injective.
		
(ii)$\Rightarrow$(i) 
Because $M$ is $\Sigma$-pure-injective, for any set $I$ the canonical morphism $\bigoplus_{i\in I} M\to \prod_{i\in I} M$ is a splitting monomorphism. 
	\end{proof}

\begin{remark}\label{rem:saroch-example}
 It is proved in \cite[Corollary 3.5]{Saroch2015} that if we add some set theoretic conditions that allow us to have $\omega$-measurable cardinals, the implication (i)$\Rightarrow$(ii) from Theorem \ref{thm:lfp-Add-subset-Prod} is not true. 
 Therefore, this theorem cannot be proved in ZFC.  
\end{remark}

As an application, we will prove that if we assume that there are no $\omega$-measurable cardinals then the converse of \cite[Proposition 6.10]{Angeleri2000} is valid. Recall that a class $\CC$ of objects in $\CA$ is called \textit{precovering} if for every object $M\in \CA$ there exists a morphism $f:C\to M$ with $C\in \CC$ such that all morphisms $g:C'\to M$ with $C'\in \CC$ can be decomposed as $g=fh$ with $h:C'\to C$. 
In this case $f$ is a \textit{$\CC$-precover} for $M$. If $f$ has the property that for every decomposition $f=fh$ we obtain that $h$ is an automorphism of $C$ then $f$ is called a \textit{$\CC$-cover} for $M$. 
If all objects from $\CA$ have $\CC$-(pre)covers, we will say that $\CC$ is a \textit{(pre)covering class}. 

The following lemma is known for the module categories. The proof can be easily adapted for idempotent complete additive categories.

\begin{lemma}\label{lem:Prod-precover} \cite[Proposition 1.2]{HJ-08}
If $\CA$ is an idempotent complete additive category with direct sums and $\CC$ is a precovering class then $\CC$ is closed under direct sums.    
\end{lemma}

\begin{lemma}\label{lem:summads-P(A)} Let $A\in \CA$ be an object and let $\overline{X}$ be a direct summand of $ev(A)$. Then there exists a direct summand $X$ of $A$ such that $ev(X)$ and $\overline{X}$ represent the same quotient of $A$.     
\end{lemma}

\begin{proof}
Since $\overline{X}$ is a direct summand of $ev(A)$, there exists an idempotent endomorphism $\overline{e}$ of $ev(A)$ such that $\overline{X}$ is the cokernel of $\overline{e}$. There exists an idempotent endomorphism $e$ of $A$ such that $\overline{e}=ev(e)$. Because $\CA$ is an idempotent complete category, $e$ has a cokernel $X$. Since $ev$ preserves the cokernels, it follows that $ev(X)$ and $\overline{X}$ are isomorphic as quotient objects of $A$.     
\end{proof}

\begin{corollary}\label{cor:enoch-fp-presented}\label{cor:angeleri-6.11}
Suppose that $\CA$ is a locally finitely presented category. If $M\in \CA$ is an object such that for every finitely presented object $X$ the cardinal of $\Hom(X,M)$ is not $\omega$-measurable, the following are equivalent:
\begin{enumerate}[{\rm (i)}]
    \item $M$ is $\Sigma$-pure injective;
    \item $\Prod(M)$ is a precovering class;
    \item $\Prod(M)$ is a covering class.
\end{enumerate}
\end{corollary}

\begin{proof}
(i)$\Rightarrow$(ii) From (i) it follows that $\Prod(M)$ is closed under direct sums and pure quotients (since every pure monomorhism splits). 
From \cite[Theorem 2.4]{LV} we obtain that $\Prod(M)$ is a precovering class.

(ii)$\Rightarrow$(i) We apply Lemma \ref{lem:Prod-precover} to obtain $\Add(M)\subseteq \Prod(M)$. The conclusion follows from Theorem \ref{thm:lfp-Add-subset-Prod}. 

(ii)$\Rightarrow$(iii) If $\CA$ is abelian, then we can apply \cite[Corollary 2.7]{LV} and Corollary \ref{cor:closure-abelian} to conclude that $\Prod(M)$ is a covering class.

In the general case, it is not hard to see that if $\Add(M)\subseteq \Prod(M)$ then $\Add(ev(M))\subseteq \Prod(ev(M))$. It follows that the class $\Prod(ev(M))$ is covering in $\mathbf{P}(\CA)$. If $X\in \CA$ and $\overline{\alpha}:\overline{K}\to ev(X)$ is an $\Prod(ev(M))$-cover, from Lemma \ref{lem:summads-P(A)} we can assume that there exists $K\in \CA$ such that $ev(K)=\overline{K}$. Since $ev$ is a fully faithful functor, we observe that there exists $\alpha:K\to X$ such that $ev(\alpha)=\overline{\alpha}$ and that $\alpha$ is a $\Prod(M)$-cover for $X$.

(iii)$\Rightarrow$(ii) This is obvious.
\end{proof}    

A well-known conjecture of Enochs states that in cocomplete abelian categories \textit{all covering classes are closed under direct limits}. Positive solutions for this conjecture are presented in \cite[Theorem 4.4]{angeleri2003} for classes of the form $\Add(M)$ where $M$ is a direct sum of finitely presented modules and in \cite[Theorem 2.2]{Sar-Enochs} for general classes of the form $\Add(M)$, but under some set theoretic hypotheses. The following is a consequence of Corollary \ref{cor:enoch-fp-presented}.

\begin{corollary}\label{cor:enochs-prod-lfp}
Assume that there are no $\omega$-measurable cardinals. If $\CA$ is a locally finitely presented category and $M\in \CA$ such that $\Prod(M)$ is a (pre)covering class then $\Prod(M)$ is closed under directed limits.    
\end{corollary}

\subsection{Product-complete objects}
	
We recall that an object $M$ of $\mathcal{A}$ is called \textit{product-complete} if $\Prod(M)\subseteq \Add(M)$. We will provide a connection between $\Sigma$-pure-injective objects and product-complete objects in $\mathcal{A}$ that extends to locally finite presented categories the characterization proved by Angeleri-H\" ugel in \cite[Proposition 4.8]{Angeleri2000}. 
	
\begin{lemma}\label{lemmasummand}
If $(M_j)_{j\in J}$ is a family of objects from $\mathcal{A}$ with local endomorphism rings, then for any direct summand $X$ of $\bigoplus_{j\in J}M_j$ with $\End(X)$ local there exists $j\in J$ such that $X\cong M_j$.
\end{lemma}
	\begin{proof}
We will use the functor $ev:\mathcal{A}\to \textbf{P}(\mathcal{A})$.		
Let $X$ a direct summand of $\bigoplus_{j\in J}M_j$ such that the endomorphism ring of $X$ is local. Then $ev(X)$ is a direct summand of $ev(\bigoplus_{j\in J}M_j)\simeq \bigoplus_{j\in J}ev(M_j)$. Moreover, the objects $ev(X)$ and $ev(M_j)$, $j\in J$, have local endomorphism rings. Due to a result of Walker and Warfield \cite[Theorem 2]{Walker&Warfield1976}, it follows that $ev(X)\simeq ev(M_j)$, for some $j\in J$. Because the functor $ev$ is fully faithful, we conclude that there exists $j\in J$ such that $X\simeq M_j$.
	\end{proof}
	
An object $M$ in an additive category with products is called \textit{product-rigid} if for every object $X\in \Prod(M)$ with local endomorphism ring, $X$ is isomorphic to a direct summand of $M$.
	
\begin{theorem}\label{thm:productcompleteA}
Let $M$ be an object of $\mathcal{A}$. The following are equivalent:
		\begin{enumerate}[{\rm (i)}]
			\item $\Prod(M) = \Add(M)$.
			
			\item $\Prod(M)\subseteq \Add(M)$.
			
			\item \begin{enumerate}[{\rm (a)}]
				\item $M$ is $\Sigma$-pure-injective.
				\item $M$ is product-rigid.
			\end{enumerate}
		\end{enumerate}
\end{theorem}
	\begin{proof}
(i)$\Rightarrow$(ii) It is obvious.
		
(ii)$\Rightarrow$(iii) (a) Since $\Prod(M)\subseteq \Add(M)$, it follows that every product of copies of $M$ is a direct summand in a direct sum of copies of $M$. Using Theorem \ref{theoremK}, we obtain that $M$ is a $\Sigma$-pure-injective object.
		
(b) Let $X\in \Prod(M)$ be an object with local endomorphism ring. By (ii) we have that $X\in \Add(M)$. Also, since $M$ is $\Sigma$-pure-injective, we can use Theorem \ref{theoremK} to see that $M$ is isomorphic to a direct sum $\bigoplus_{j\in J} X_j$ of objects with local endomorphism rings. Thus, we obtain that $X$ is a direct summand of a direct sum $\bigoplus_{j\in J} X_j^{(I)}$ of objects with local endomorphism rings. According to Lemma \ref{lemmasummand}, it follows that $X$ is isomorphic to one of the objects $X_j$.
		
(iii)$\Rightarrow$(i) Since for $A$ is $\Sigma$-pure-injective, we have $\Add(M)\subseteq \Prod(M)$, and it is enough to prove that $\Prod(M)\subseteq \Add(M)$.
		
Suppose that $X$ is a direct summand of a direct product $\prod_{i\in I} M$.		
Because $M$ is $\Sigma$-pure-injective, $\prod_{i\in I} M$ is isomorphic to a direct sum of indecomposable objects with local endomorphism rings (see Theorem \ref{theoremK}, implication (i)$\Rightarrow$(iii)), so $\prod_{i\in I} M\simeq \bigoplus_{j\in J} X_j$, where $X_j\in \Prod(M)$ are objects with local endomorphism rings.
		
But $M$ is product-rigid, hence every $X_j$ is isomorphic to a direct summand of $M$. 
It follows that $\prod_{i\in I} M\in \Add(M)$, hence $X\in \Add(M)$.
	\end{proof}

As an application, we obtain the dual of Corollary \ref{cor:angeleri-6.11}. This generalizes \cite[Theorem 5.1]{angeleri2003} and \cite[Theorem 2.3]{Kr-Sa} (see also \cite[Proposition 13.54]{GoTr-12}) from module categories to locally finitely presented categories.

\begin{corollary}\label{cor:add-env}
The following are equivalent for an object $M\in\CA$:
\begin{enumerate}[{\rm (i)}]
    \item $M$ is product-complete;
    \item $\Add(M)$ is a preenveloping class in $\CA$;
    \item $\Add(M)$ is an enveloping class. 
\end{enumerate}
\end{corollary}
\begin{proof}
i)$\Rightarrow$ii) This follows from the equality $\Add(M)=\Prod(M)$, since all classes of the form $\Prod(M)$ are preenveloping. 

ii)$\Rightarrow$i) We can use the dual of Lemma \ref{lem:Prod-precover} to conclude that $\Prod(M)\subseteq \Add(M)$.

ii)$\Rightarrow$iii) The proof is the same as that of \cite[Theorem 2.3]{Kr-Sa}, using this time the evaluation functor $ev:\CA\to\mathbf{P}(\CA)$ together with \cite[Proposition 1.2]{Kr-Sa}.
\end{proof}
    
	\section{The case of compactly generated triangulated categories}
In this section, $\mathcal{T}$ will be a compactly generated triangulated category. This means that $\calT$ is triangulated, has all direct sums (coproducts), the triangulated subcategory $\mathcal{T}^c$ of all compact objects from $\mathcal{T}$ ($C$ is compact if $\Hom(C,-)$ commutes with respect to all direct sums) is small and $\Hom(\calT^c,X)\neq 0$ for all nonzero objects $X\in \calT$. We recall that $\calT$ also has products.

\subsection{Purity in compactly generated triangulated categories} We recall here, for further use, some basic information about purity in compactly generated triangulated categories.
 We denote by $\Modr\mathcal{T}^c$ the category of functors $(\mathcal{T}^c)^{op}\to \textbf{Ab}$, a locally finitely presented Grothendieck category.

There is a restricted Yoneda functor, introduced in \cite{Krause2000}, 
$$H:\calT\to \Modr\mathcal{T}^c,\ M\mapsto H_M = \Hom(-,M)|_{\mathcal{T}^c}.$$
A triangle $A\overset{\alpha}\to B\overset{\beta}\to C\overset{\phi}\to A[1]$ is \textit{pure} if $\Hom(\calT^c,\phi)=0$. This is equivalent to the fact that the induced sequence $0\to H_A\to H_B\to H_C\to 0$ is a short exact sequence in $\Modr\calT^c$. In this case, we will say that $\alpha$ is \textit{pure monomorphism} and $\beta$ is \textit{pure epimorphism}, $A$ is a \textit{pure subobject} of $B$ and $C$ is a \textit{pure quotient} of $B$. An object $M$ is \textit{pure-injective} if $\Hom(\alpha,M)$ is surjective for all pure monomorphisms $\alpha$ (i.e., $\Hom(-,M)$ transforms pure triangles in exact sequences of abelian groups).   

The equivalences between (i), (ii) and (iv) in the following theorem were proved in \cite[Theorem 1.8]{Krause2000}. The equivalence between (ii) and (iii) was observed in \cite[Section 1.B]{Prest23}.
	
\begin{theorem}\label{Krausepureinjective}
The following are equivalent for an object $M\in \mathcal{T}$:
		\begin{enumerate}[{\rm (i)}]
			\item $M$ is pure-injective;
			\item $H_M = \Hom(-,M)|_{\mathcal{T}^c}$ is injective in $\Modr\mathcal{T}^c$;
            \item $H_M = \Hom(-,M)|_{\mathcal{T}^c}$ is pure-injective in $\Modr\mathcal{T}^c$;
			\item for every $X$ in $\mathcal{T}$ the map $\Hom(X,M)\to \Hom(H_X,H_M), \phi\mapsto \Hom(-,\phi)|_{\mathcal{T}^c}$ is an isomorphism.
		\end{enumerate}
\end{theorem}

\subsection{$\Sigma$-pure-injective objects in compactly generated triangulated categories and ${\Add(M)\subseteq \Prod(M)}$} 	
An object $M$ is \textit{$\Sigma$-pure-injective} if all direct sums of copies of $M$ are pure-injective. 	
Since $H_{M^{(J)}}\simeq (H_M)^{(J)}$, we have that $M\in \mathcal{T}$ is $\Sigma$-pure-injective if and only if $H_M$ is $\Sigma$-pure-injective in $\Modr\mathcal{T}^c$. Some of the characterizations of $\Sigma$-pure-injective objects in locally finitely presented categories have similar correspondents for $\Sigma$-pure-injective objects in compactly generated triangulated categories, \cite{Bennett-Tennenhaus2023}. We will use the following theorem.
	
\begin{theorem}\cite[Theorem 1]{Bennett-Tennenhaus2023}\label{theoremBT} Let $M\in\mathcal{T}$. The following are equivalent:
		\begin{enumerate}[{\rm (i)}]
			\item $M$ is $\Sigma$-pure-injective.
			\item $M$ is pure-injective and for any set $I$, the product $M^I$ is isomorphic to a direct sum of indecomposable pure-injective objects with local endomorphism rings.
		\end{enumerate}
	\end{theorem}

We have a corollary similar with Corollary \ref{cor:closure-abelian}.

\begin{corollary}\label{cor:closure-triangulated}\cite[Corollary 6.10]{Bennett-Tennenhaus2023} 
Suppose that $M$ is a $\Sigma$-pure-injective object from $\calT$. The following are true:
\begin{enumerate}[{\rm a)}]
    \item For every set $I$ the object $M^I$ is $\Sigma$-pure-injective. 

\item If $\alpha:K\to M$ is a pure monomorphism then $K$ is $\Sigma$-pure-injective and $\alpha$ is a splitting monomorphism.

\item Every pure epimorphism $M\to N$ is a splitting epimorphism.   
\end{enumerate}
\end{corollary}

In compactly generated triangulated categories, there is a similar characterization as that presented in Theorem \ref{thm:lfp-Add-subset-Prod} for the case of locally finitely presented categories.

\begin{theorem}\label{thm:triang-Add-subset-Prod}
If $M$ is an object of $\mathcal{T}$ such that for every compact $C$ the cardinal of $\Hom(C,M)$ is not $\omega$-measurable, the following are equivalent:
		\begin{enumerate}[{\rm (i)}]
			\item $\Add(M)\subseteq \Prod(M)$;
            \item $\Add(H_M)\subseteq \Prod(H_M)$;
			\item $M$ is $\Sigma$-pure-injective.
		\end{enumerate}
\end{theorem}
	
\begin{proof}
(i)$\Rightarrow $(ii) It is enough to show that for every set $I$ the direct sum $(H_M)^{(I)}$ belongs to $\Prod(H_M)$. This is true since $(H_M)^{(I)}\cong H_{M^{(I)}}$, and we have a splitting monomorphism $M^{(I)}\to M^J$ for some set $J$. 
        
(ii)$\Rightarrow$(iii) From Yoneda's Lemma it follows that for every compact $C$ the cardinal of the homomorphism group $\Hom_{\Modr\mathcal{T}^c}(H_C,H_M)$ is not $\omega$-measurable. From \cite[Theorem 2.11]{Ek-Me}, it follows that for every finitely presented object $K$ from $\Modr\mathcal{T}^c$ the group $\Hom_{\Modr\mathcal{T}^c}(K,H_M)$ is not $\omega$-measurable.
        
Because $\Modr\mathcal{T}^c$ is a locally finitely presented category, applying the results of the previous section, it follows that $H_M$ is $\Sigma$-pure-injective. Then $M$ is a $\Sigma$-pure-injective object in $\mathcal{T}$.
		
(iii)$\Rightarrow $(i) It is clear.
\end{proof}

\begin{remark}
For the implication (i)$\Rightarrow$(ii) we do not need the set theoretic assumption. The implications (i)$\Rightarrow$(iii) and (ii)$\Rightarrow$(iii) do not hold in ZFC. In order to see this, it is enough to consider one more time \v Saroch's example \cite[Corollary 3.5]{Saroch2015}. Let us observe that if $R$ is a ring then the standard embedding of $\Modr R$ in its derived category $\mathbf{D}(R)$ is fully faithful and it preserves all direct sums and direct products. Since in both categories $\Modr R$ and $\mathbf{D}(R)$ an object $X$ is pure-injective if and only if for every set $I$ the summation morphism $X^{(I)}\to X$ can be extended to a morphism $X^I\to X$, it follows that an $R$-module $X$ is pure-injective in $\Modr R$ if and only if it is pure-injective in $\mathbf{D}(R)$. Therefore, if there exists a module $M$ such that $M$ is not pure-injective and $\Add(M)\subseteq \Prod(M)$, then $M$ has the same properties in $\mathbf{D}(R)$.          \end{remark}

As in the case of locally finitely presented categories, the $\Sigma$-pure-injective injective objects can be characterized by some approximation properties of $\Prod(M)$. Since in this case the functor $H$ is not fully faithful, we have to use some alternative techniques based on the following results. Recall that a ring $R$ is \textit{semiregular} (or \textit{f-semiperfect}) if all idempotents lift modulo $J(R)$ and $R/J(R)$ is von Neumann regular.
 
\begin{lemma}\label{lem:end-pure-injective} 
If $M$ is a pure-injective object in $\calT$ then the endomorphism ring of $M$ is semiregular.     
\end{lemma}

\begin{proof}
We know from \cite[Corollary 1.6]{Simson78} that in every locally finitely presented Grothendieck category the endomorphism rings of pure-injective objects are semiregular. The conclusion follows from Theorem \ref{Krausepureinjective}.     
\end{proof}

The following lemma is known for module categories. The proof presented in \cite[Lemma 5.9]{GoTr-12} can be extended to a general setting.

\begin{lemma}\label{lem:Go-Tr5.9}
Let $\CA$ be an idempotent complete additive category. Suppose that $\CC$ is a class of objects from $\CA$ closed under finite direct sums and direct summands and that $f:C\to A$ is a $\CC$-precover for $A$. If 
\begin{enumerate}[{\rm (i)}]
    \item the idempotents lift modulo the Jacobson radical of $\End(C)$, and
        \item for the right ideal $I=\{g\in \End(C)\mid fg=0\}$ there exists a right ideal $J$ such that $I+J=\End(C)$ and $I\cap J$ is contained in the radical of $\End(C)$
\end{enumerate}
then $M$ has a $\CC$-cover. 
\end{lemma}

\begin{corollary}\label{cor:enoch-triang}
Suppose that $\calT$ is an algebraic compactly generated triangulated category and $M\in \calT$ is an object such that for every compact $C$ the set $\Hom(C,M)$ is not $\omega$-measurable. The following are equivalent:
\begin{enumerate}[{\rm (i)}]
    \item $M$ is $\Sigma$-pure-injective;
  \item $\Prod(M)$ is definable;
    \item $\Prod(M)$ is a precovering class;
  \item $\Prod(M)$ is a covering class.
\end{enumerate}
\end{corollary}

\begin{proof} 
(i)$\Rightarrow$(ii) This follows from \cite[Theorem 4.7]{LV} and Corollary \ref{cor:closure-triangulated}.

(ii)$\Rightarrow$(iii) It is enough to apply \cite[Corollary 4.8]{LV}.

(iii)$\Rightarrow$(i) It is a consequence of Lemma \ref{lem:Prod-precover} and Theorem \ref{thm:triang-Add-subset-Prod}. 

(iv)$\Rightarrow$(iii) This is obvious.

(iii)$\Rightarrow$(iv) We will follow the main ideas presented in the proof of \cite[Proposition 4.1]{angeleri2003}.

Let $A$ be an object from $\calT$. It is not hard to see that it has a $\Prod(M)$-precover of the form $\alpha:M^I\to A$. We complete $\alpha$ to a triangle $B\overset{\gamma}\to M^I\overset{\alpha}\to A\to B[1] ,$ and consider a $\Prod(M)$-precover $\beta:M^J\to B$. If $|J|>|I|$, we can assume $I\subseteq J$ and that $J$ is infinite. We write $M^J=M^I\oplus M^{J\setminus I}$ and we have the triangle 
$$ B\oplus M^{J\setminus I}\overset{\gamma\oplus 1}\longrightarrow M^I\oplus M^{J\setminus I}\overset{\alpha\oplus 0}\longrightarrow A\to (B\oplus M^{J\setminus I})[1]. $$
Observe that $\alpha\oplus 0$ is a $\Prod(M)$-precover for $A$ and that $\beta\oplus \pi:M^J\oplus M^J\to B\oplus M^{J\setminus I}$, where $\pi$ is the canonical projection, is a $\Prod(M)$-precover for $B\oplus M^{J\setminus I}$. 

It follows that we always have a triangle 
$$C\overset{\gamma}\to M^{J}\overset{\alpha}\to A\to C[1] \text{ and a morphism } \beta:M^K\to C$$
such that $\alpha$ and $\beta$ are $\Prod(M)$-precovers for $A$ and $C$, respectively, and $|K|\leq |J|$. Applying the functor $\Hom(M^J,-):\calT\to \Modr\End(M^J)$, we obtain the exact sequences of $\End(M^J)$-modules:
$$\Hom(M^J,C)\to \Hom(M^J,M^{J})\to \Hom(M^J,A)\to 0$$ \text{ and } $$\Hom(M^J,M^K)\to \Hom(M^J,C)\to 0.$$ 
It follows that $\Ker(\Hom(M^J,\alpha))=\{\delta\in\End(M^J)\mid \alpha\delta=0\}$ is a finitely generated right ideal of $\End(M^J)$. Using Corollary \ref{cor:closure-triangulated} and Lemma \ref{lem:end-pure-injective} we observe that $\End(M^J)$ is semiregular, hence every finitely generated right ideal is a direct summand modulo the Jacobson radical. We apply Lemma \ref{lem:Go-Tr5.9} to conclude that $A$ has a $\Prod(M)$-cover.  
\end{proof}

We can formulate a triangulated version of Enochs' conjecture: \textit{every covering class in an algebraic compactly generated triangulated category is closed under directed homotopy colimits.}

\begin{corollary}\label{cor:Enochs-triangulated}
Assume that there are no $\omega$-measurable cardinals. Then the triangulated Enochs's conjecture is valid for classes of the form $\Prod(M)$ in algebraic compactly generated triangulated categories.      
\end{corollary}

\begin{proof}
Suppose that $\calT$ is an algebraic compactly generated triangulated category. From the proof of \cite[Proposition 6.8]{BW} it follows that if $X_i$ is a directed family in $\calT$ then the canonical morphism $\oplus_{i\in I}X_i\to\mathrm{hocolim}_{i\in I}X_i$ is a pure epimorphism. Using Corollary \ref{cor:enoch-triang}, we obtain the conclusion.
\end{proof}

\subsection{Product-complete objects}
As before, we have a triangulated version for the characterization of product-complete objects in locally finitely presented categories. 

For information on underlying categories of strong and stable derivators, we refer to \cite{La}. 
  
\begin{theorem}\label{thm:triang-Prod-subset-Add}
Let $M$ be an object of $\mathcal{T}$. The following are equivalent:
\begin{enumerate}[{\rm (i)}]
			\item $\Prod(M) = \Add(M)$;
			
			\item $\Prod(M)\subseteq \Add(M)$;
			
			\item $\Prod(H_M)\subseteq \Add(H_M)$;
			 
			\item \begin{enumerate}[{\rm (a)}]
				\item $M$ is $\Sigma$-pure-injective;
				\item $M$ is product-rigid.
			\end{enumerate}
\end{enumerate}
        
Moreover, if $\calT$ is the underlying category of a strong and stable derivator the above conditions are equivalent with 
   \begin{itemize}
            \item[{\rm (v)}] $\Add(M)$ is definable;
            \item[{\rm (vi)}] \begin{enumerate}[{\rm (a)}] 
            \item $M$ is $\Sigma$-pure-injective,
            \item the set of all indecomposable direct summands of $M$ form a Ziegler closed set.
            \end{enumerate}  
        \end{itemize}
	\end{theorem}
\begin{proof}
(i)$\Rightarrow$ (ii) is obvious. 
        
(ii)$\Rightarrow $(iii) This follows by using steps similar to those presented in Theorem \ref{thm:triang-Add-subset-Prod}.
		
(iii)$\Rightarrow$ (iv) We have $\Prod(H_M)\subseteq \Add(H_M)$, and we use Theorem \ref{thm:productcompleteA} to observe that $H_M$ is $\Sigma$-pure-injective and product-rigid. Since $H_M$ is $\Sigma$-pure-injective, we also obtain that $M$ is $\Sigma$-pure-injective.
		
If $X\in \Prod(M)$ is an object with local endomorphism ring then $X$ is pure-injective, hence $\Hom(X,X)\simeq \Hom(H_X,H_X)$. Then the endomorphism ring of $H_X$ is local, and we obtain that $H_X$ is isomorphic to a direct summand of $H_M$. Using Theorem \ref{Krausepureinjective} one more time, it follows that $X$ is isomorphic to a direct summand of $M$.
		
		
		(iv)$\Rightarrow$ (i) Since $\Add(M)\subseteq \Prod(M)$ for any $\Sigma$-pure-injective object $M$, we have to prove the inclusion $\Prod(M)\subseteq \Add(M)$. In fact it is enough to prove that  $\prod_{i\in I} M\in\Add(M)$ for any set $I$.
		
Because $M$ is $\Sigma$-pure-injective, $\prod_{i\in I} M$ is isomorphic to a direct sum of indecomposable pure-injective objects with local endomorphism rings (see Theorem \ref{theoremBT}), so $\prod_{i\in I} M\simeq \bigoplus_{j\in J} X_j$ such that $X_j$ are objects with local endomorphism rings from $\Prod(M)$.
Since $M$ is product-rigid, every $X_j$ is isomorphic to a direct summand of $M$, 
so $\prod_{i\in I} M\in \Add(M)$. 

(i)$\Rightarrow$(v) From Corollary \ref{cor:closure-triangulated}, it follows that the class $\Add(M)$ is closed under pure subobjects, products and pure quotients. Now we use \cite[Proposition 6.8]{BW} to conclude that $\Add(M)$ is a definable class. 

The proofs for (v)$\Rightarrow$(vi) and (vi)$\Rightarrow$(i) are the same as those presented for finitely presented categories in \cite[Proposition 12.3.7]{Krause2022}, this time using the fundamental correspondence presented in \cite{Krause2002}.
\end{proof}
	
\begin{remark}
The equivalence (ii)$\Leftrightarrow$(iii) says that an object $M$ is product-complete in $\mathcal{T}$ if and only if $H_M$ is product-complete in $\Modr\mathcal{T}^c$ and the equivalence (iii)$\Leftrightarrow$(iv) also says that, under the hypothesis of $\Sigma$-pure-injectivity on the object $M$, $M$ is product-rigid in $\mathcal{T}$ if and only if $H_M$ is product-rigid in $\Modr\mathcal{T}^c$. We do not know if this correspondence between product-rigid objects holds in the absence of the pure-injectivity hypothesis.
\end{remark}

We also have a triangulated version, with a similar proof, of Corollary \ref{cor:add-env}. 

\begin{corollary}
The following are equivalent for an object $M\in\calT$:
\begin{enumerate}[{\rm (i)}]
    \item $M$ is product-complete;
    \item $\Add(M)$ is a preenveloping class in $\CA$;
    \item $\Add(M)$ is an enveloping class. 
\end{enumerate}
\end{corollary}

\section*{Acknowledgements}
C. Rafiliu was supported by the project CNFIS-FDI-2024-F-0456.

\end{document}